\documentclass[11pt,reqno]{amsart}

\usepackage{tikz-cd}
\usepackage{amsfonts, amssymb, amscd}
\usepackage{amsmath}
\usepackage{verbatim}
\usepackage{amssymb}
\usepackage{mathrsfs}
\usepackage{graphicx}
\usepackage{cite}
\usepackage[all]{xy}
\usepackage{tikz}
\usepackage{subfiles}
\usepackage[toc,page]{appendix}
\usepackage{filecontents}
\usepackage{tikz-cd}
\usepackage{graphicx}
\usepackage{hyperref}
\usepackage{float}
\usepackage{color}

\def\C {{\mathbb C}}

\def\R {{\mathbb R}}
\def\Z {{\mathbb Z}}

\allowdisplaybreaks[3]

\usetikzlibrary{arrows.meta,decorations.markings}

\tikzset{
    midarrow/.style={
    postaction={decorate},
    decoration={markings, mark=at position #1 with {\arrow{Stealth}}}
    },
    midarrow/.default=0.55
}

\theoremstyle{definition}

\newtheorem{theorem}{Theorem}[section]
\newtheorem{lemma}[theorem]{Lemma}
\newtheorem{proposition}[theorem]{Proposition}
\newtheorem{definition}[theorem]{Definition}
\newtheorem{example}[theorem]{Example}

\newtheorem{remark}[theorem]{Remark}

\numberwithin{equation}{section}

\begin{document}

\title{Hanlon-Hicks-Lazarev resolution revisited}

\begin{abstract}
    Hanlon, Hicks and Lazarev constructed resolutions of structure sheaves of toric substacks by certain line bundles on the ambient toric stacks. In this paper, we give a new and substantially simpler proof of their result.
\end{abstract}

\author{Lev Borisov}
\address{Department of Mathematics\\
Rutgers University\\
Piscataway, NJ 08854} \email{borisov@math.rutgers.edu}

\author{Zengrui Han}
\address{Department of Mathematics\\
University of Maryland\\
College Park, MD 20742} \email{zhan223@umd.edu}

\maketitle
\tableofcontents

\section{Introduction}

Mirror symmetry for toric varieties has been extensively studied  over the last thirty years, going back to the pioneering paper of Batyrev \cite{Batdual}. In particular, there has been a lot of work on bounded derived categories of toric varieties, inspired by the Homological Mirror Symmetry conjectures of Kontsevich.
It  was first observed by Bondal  in \cite{Bondal}, that these  categories are closely related to the derived categories of  constructible sheaves on certain real tori  with respect to stratifications induced by toric hyperplanes defined by ray generators of the fan. This perspective was further developed by Fang, Liu, Treumann, and Zaslow \cite{FLTZ} who related torus-equivariant coherent sheaves on the toric varieties to constructible sheaves on the real tori with microsupport on conic Lagrangian subsets of the cotangent bundles of the real tori called FLTZ skeleta. 
These results recently inspired Hanlon, Hicks and Lazarev \cite{HHL} to construct resolutions of structure sheaves of toric substacks in smooth toric stacks by line bundles on the ambient stacks, which we call Hanlon-Hicks-Lazarev resolutions, or HHL resolutions for short. Their construction has in turn already inspired further papers by multiple authors, see \cite{BBB+,BE,FS,Heller}.

\smallskip

As far as we understand it, the original argument in \cite{HHL} is based on a hyperplane removal strategy. By using discrete Morse theory, the HHL complex can be proven to be homotopy equivalent to the complex associated to the coarser stratification obtained by removing one toric hyperplane on an open subset. Then one keeps removing hyperplanes until arriving at a Koszul-type stratification, thereby obtaining a homotopy equivalence to the standard Koszul resolution on (an open subset) of an affine chart. Running this process for all affine charts proves the result over the complement of a codimension two subset. One then needs extra effort to extend to the whole variety or stack.

\smallskip

The main goal of this paper is twofold. First, we provide a new and substantially simpler proof of the main result of \cite{HHL}. In particular, our argument avoids induction on the maximal cones, the hyperplane removal strategy and the discrete Morse machinery. Second, we restrict to a minimal amount of input data, thereby making our setup cleaner and more concise. We hope this makes this topic more accessible and that our paper could serve as a more user-friendly introduction. 

\smallskip

Now we explain the strategy of our proof and the organization of this paper. In Section \ref{sec-affine-HHL}, we study the HHL complexes on affine spaces. This is a purely commutative-algebra construction that requires little knowledge of toric geometry. 

\smallskip
More precisely, let  $L\cong \Z^n$ be a lattice of rank $n$ with a $\Z$-basis $\{v_i\}$ and let $\psi:L\rightarrow\Lambda$ 
 be a morphism from $L$ to another lattice $\Lambda$ of rank $k$ such that $\mathrm{coker}\,\psi$ is finite.
We then define a complex of free modules over the polynomial ring $R= \C[x_1,\cdots,x_n]$ which we call the affine HHL complex. Our main result is the following Theorem \ref{thm-1}, where the semigroup algebra $\C[C\cap M]$ is defined in Section \ref{sec-affine-HHL}.

\smallskip
\smallskip\noindent
{\bf Theorem \ref{thm-1}.}
    The only nonzero homology of the affine HHL complex \eqref{eq-affine-HHL-complex} is $H_0=\C[C\cap M]$. 

\smallskip

The key ingredient of the proof is a decomposition of the affine HHL complex into a direct sum of complexes of $\C$-vector spaces, where each piece is isomorphic to the cellular chain complex of some polyhedral region on the universal cover of the real torus. 

\smallskip
To get more general HHL complexes, as in \cite{HHL}, it remains to pass from modules to sheaves on $\C^n$, then restrict to toric open subsets $U_\Sigma$ of $\C^n$ and finally use the equivariant nature of the construction to pass to quotients of $U_\Sigma$ by abelian reductive groups, see Theorem \ref{thm-2} and Theorem  \ref{thm-3} in Section \ref{sec-general-HHL}.

\smallskip

We end the introduction with a discussion on related works. Our main result settles a conjecture raised by Michael Brown and Daniel Erman in \cite{BE} regarding the relationship between the resolution constructed therein and the HHL resolution, see Remark \ref{rmk-BE-conjecture} for details. It also seems plausible that our method can be used to give a cleaner proof of the topological formula for the Betti numbers of the minimization of the HHL resolutions in \cite{FS}. For a more general discussion of cellular free resolutions, the reader is referred to \cite{BCHSY}.

\smallskip

{\bf Acknowledgements.} We thank Michael Brown, Daniel Erman and Andrew Hanlon for useful conversations and valuable suggestions regarding an earlier version of this paper. The first author was also inspired by an Oberwolfach talk by Christine Berkesch.

\section{HHL complexes on affine spaces}\label{sec-affine-HHL}

In this section we study the affine case, i.e., HHL complexes that are defined on affine spaces $\C^n$. The context of this section should be accessible and interesting to commutative algebraists or readers who are not familiar with toric geometry. By the homogeneous coordinate ring construction, the general HHL result for toric stacks follows almost effortlessly from this seemingly special situation.

\subsection{Affine HHL complex}\label{subsec-affine-HHL}

We start with a minimal amount of input data. Consider a pair $(L,\psi:L\rightarrow\Lambda)$ where:
\begin{itemize}
    \item $L\cong \Z^n$ is a lattice of rank $n$ with a $\Z$-basis $\{v_i\}$,
    \item $\psi:L\rightarrow\Lambda$ is a morphism from $L$ to another lattice $\Lambda$ of rank $k$ such that $\mathrm{coker}\,\psi$ is finite.
\end{itemize}

\begin{definition}
    We define stratifications on the real torus
    \begin{align*}
        T := \operatorname{Hom}(\Lambda, \R/\Z) = \Lambda^*_{\R} / \Lambda^* \cong (S^1)^k
    \end{align*}
    and its universal cover
    \begin{align*}
        \Lambda^*_{\R} = \operatorname{Hom}(\Lambda, \R)  \cong \R^k
    \end{align*}
    which we call \textit{Bondal stratifications}.
\end{definition}

Each basis vector $v_i$ of $L$ defines a map
\begin{align*}
    H_{i}:\ \Lambda^*_{\R} \rightarrow\R,\quad f\mapsto f(\psi(v_i)).
\end{align*}
The hyperplanes in $\Lambda^*_{\R}$ defined by $H_i = a$ for $a\in \Z$ together define a periodic hyperplane arrangement on the space $\Lambda^*_{\R}$, and we denote the induced stratification by $\widetilde{S}$. Note that they naturally descend to the torus $T$ thereby inducing a stratification there which we denote by $S$. We denote the set of $m$-dimensional strata in $S$ and $\widetilde{S}$ by $S_m$ and $\widetilde{S_m}$, respectively. The fact that $\psi$ has finite cokernel implies that each stratum $\widetilde{\sigma}$ of $\widetilde{S}$ is a bounded convex polytope. For a stratum $\widetilde{\sigma}$ in $\widetilde{S}$, we denote by $H_{i}(\widetilde{\sigma})$ the value of $H_{i}$ at an arbitrary point in the relative interior of $\widetilde{\sigma}$. Note that $\left\lceil H_{i}(\widetilde{\sigma})  \right\rceil $ is well-defined and does not depend on the choice of the point.

\smallskip

We denote the coordinate ring of $\C^n$ by $R=\C[x_1,\cdots,x_n]$. From the stratification $S$ we define a complex of free $R$-modules
\begin{align}\label{eq-affine-HHL-complex}
    0\rightarrow\bigoplus_{\substack{\sigma\in S_{k} }} R\, e_{\sigma} \rightarrow \cdots \rightarrow \bigoplus_{\substack{\sigma\in S_{1}}} R\, e_{\sigma} \rightarrow \bigoplus_{\substack{\sigma\in S_{0}}} R\, e_{\sigma}\rightarrow0,
\end{align}
which we call the \textit{affine HHL complex}, where $e_{\sigma}$ is the generator corresponding to the stratum $\sigma$. 

\smallskip

The differentials are defined as follows. We fix an arbitrary choice of orientations for all strata in the Bondal stratification on the torus $T$. Note that this also induces orientations for all lifts of the strata in the universal cover $\Lambda^*_{\R}$. Let $\sigma$ be an $m$-dimensional stratum in $S$, and $\tau$ an $(m-1)$-dimensional facet of $\sigma$. We take an arbitrary lift $\widetilde{\sigma}$ of $\sigma$ in the universal cover $\Lambda^*_{\R}$, and look at all facets of $\widetilde{\sigma}$ that are mapped to $\tau$ under the quotient map. Note that there may be multiple facets that are mapped to the same stratum on the torus. Let $\widetilde{\tau}$ be such a facet. Then for any $i$, the difference $\epsilon_{i}:=\left\lceil H_{i}(\widetilde{\sigma})  \right\rceil - \left\lceil H_{i}(\widetilde{\tau})  \right\rceil $ is either $0$ or $1$. Therefore we have a map
\begin{align*}
    R\,e_{\sigma}\rightarrow R\,e_{\tau},\quad e_{\sigma}\mapsto \left(\mathrm{sgn}(\widetilde{\sigma},\widetilde{\tau})\prod x_{i}^{\epsilon_{i}}\right) e_{\tau}
\end{align*}
where $\mathrm{sgn}(\widetilde{\sigma},\widetilde{\tau})$ equals to $1$ if the orientations on $\widetilde{\sigma}$ and $\widetilde{\tau}$ are compatible and $-1$ otherwise. The differentials in the affine HHL complex are defined as the sum of all such morphisms. It is straightforward to check that we get a complex, which will also follow from a subsequent identification with a direct sum of certain cellular complexes.

\begin{remark}
    Note that the definition of the complex does not depend on the choices of lifts, and different choices of orientations lead to isomorphic complexes.
\end{remark}

\subsection{Computation of homology}

In this subsection we compute the homology of the affine HHL complex \eqref{eq-affine-HHL-complex}. Note that each component $R\,e_{\sigma}$ is a $\C$-vector space with a basis consisting of monomials $\prod x_i^{a_i} e_{\sigma}$, where all exponents $a_i$ are nonnegative.

\begin{definition}
    We define the \textit{degree} of the monomial $\prod x_i^{a_i} e_{\sigma}$ as
\begin{align*}
    \deg\left(\prod x_i^{a_i} e_{\sigma}\right) := \sum a_i v_i^* + \sum \lceil H_i(\widetilde{\sigma}) \rceil v_i^* \  \mod\Lambda^*
\end{align*}
in the quotient $M:=L^* / \Lambda^*$, where $\widetilde{\sigma}$ is an arbitrary lift of $\sigma$ and $\{v_i^*\}$ is the dual basis to $\{v_i\}$.
\end{definition}

\begin{remark}
    The definition does not depend on the choice of the lift because we work modulo $\Lambda^*$. Also note that $M$ may have torsion, depending on whether $\psi: L\rightarrow \Lambda$ is surjective or not, see Example \ref{eg-torsion}.
\end{remark}

\begin{lemma}
    The degree of a monomial $\prod x_i^{a_i} e_{\sigma}$ remains invariant when applying the differential $d$. As a consequence, the complex \eqref{eq-affine-HHL-complex} splits into a direct sum according to the elements $l\in M$.
\end{lemma}
\begin{proof}
    Take an arbitrary lift $\widetilde{\sigma}$ of $\sigma$ and a facet $\widetilde{\tau}$ of $\widetilde{\sigma}$. By the definition of the differential, the image of the monomial $\prod x_i^{a_i} e_{\sigma}$ has a component $\prod x_i^{a_i+\epsilon_i} e_{\tau}$ where $\epsilon_{i}:=\left\lceil H_{i}(\widetilde{\sigma})  \right\rceil - \left\lceil H_{i}(\widetilde{\tau})  \right\rceil $. Therefore 
    \begin{align*}
    \deg\left(\prod x_i^{a_i} e_{\sigma}\right) &= \sum a_i v_i^* + \sum \lceil H_i(\widetilde{\sigma}) \rceil v_i^* \  \mod\Lambda^* \\
    &=  \sum (a_i+\epsilon_i) v_i^* + \sum \lceil H_i(\widetilde{\tau}) \rceil v_i^* \  \mod\Lambda^* \\
    &= \deg\left(\prod x_i^{a_i+\epsilon_i} e_{\tau}\right). \qedhere
    \end{align*}
\end{proof}

\smallskip

Now we fix an element $l\in M$. The $l$-component is explicitly written as
\begin{align}\label{eq-l-complex}
    \cdots\rightarrow\bigoplus_{\substack{ \sigma\in S_m, (a_i)\in\mathbb{N}^n \\ \deg(\prod x_i^{a_i} e_{\sigma}) = l}} \C \, \prod x_i^{a_i} e_{\sigma} \xrightarrow{d} \bigoplus_{\substack{ \tau\in S_{m-1}, (a_i)\in\mathbb{N}^n\\ \deg(\prod x_i^{a_i} e_{\tau}) = l}} \C \, \prod x_i^{a_i} e_{\tau}\rightarrow\cdots.
\end{align}
We fix an arbitrary lift $\widetilde{l} = \sum_{i=1}^n l_i v_i^*$ of $l$ in $L^*$ and define a polyhedral region 
\begin{align*}
    U_{\widetilde{l}} = \left\{ x\in \Lambda^*_{\R}\ :\ H_i(x) \leq l_i,\ \forall i=1,\cdots,n\right\}
\end{align*}
in the universal cover $\Lambda^*_{\R}$, which may be bounded, unbounded or empty. It is clear that $U_{\widetilde{l}}$ is a union of HHL strata, therefore the Bondal stratification induces a cell decomposition on it. The next proposition shows that the $l$-component of the homology\footnote{The regions $U_{\widetilde{l}}$ defined by different lifts of $l$ differ from each other by a translation in the space $\Lambda^*_{\R}$, so the isomorphism type does not depend on the lifts.} of the HHL complex is isomorphic to the singular homology of $U_{\widetilde{l}}$.

\begin{proposition}\label{prop-isomorphic-to-singular}
    The $l$-component of the HHL complex is isomorphic to the cellular chain complex of $U_{\widetilde{l}}$ induced by the Bondal stratification.
\end{proposition}
\begin{proof}
    First of all, we claim that there is a natural bijection between the set of monomials $\prod x_i^{a_i} e_{\sigma}$ with degree $l$ and the set of strata $\widetilde{\sigma}$ on the universal cover $\Lambda^*_{\R}$ that are contained in $U_{\widetilde{l}}$. To see this, recall that $\widetilde{l} = \sum l_i v_i^*$ is a fixed lift of $l$ in $L^*$. Then for each monomial $\prod x_i^{a_i} e_{\sigma}$, there is a unique lift $\widetilde{\sigma}$ such that the following equality holds in $L^*$:
    \begin{align}\label{eq1}
        \sum_{i=1}^n a_i v_i^* + \sum_{i=1}^n \lceil H_i(\widetilde{\sigma}) \rceil v_i^* = \sum_{i=1}^n l_i v_i^*.
    \end{align}
    The condition $a_i\geq 0$ then ensures $\widetilde{\sigma}\subseteq U_{\widetilde{l}}$. On the other hand, for any $\widetilde{\sigma}\in \widetilde{S_m}$ with $\widetilde{\sigma}\subseteq U_{\widetilde{l}}$, it is straightforward to check that
    \begin{align*}
        \prod x_i^{l_i - \lceil H_i(\widetilde{\sigma}) \rceil } e_{\sigma}
    \end{align*}
    corresponds to $\widetilde{\sigma}$ under the correspondence above. The condition $\widetilde{\sigma}\subseteq U_{\widetilde{l}}$ ensures all exponents are non-negative. This gives maps between the complex \eqref{eq-l-complex} and
    \begin{align*}
        \cdots\rightarrow  \bigoplus_{\substack{ \widetilde{\sigma}\in \widetilde{S_m} \\ \widetilde{\sigma}\subseteq U_{\widetilde{l}} }} \C \cdot e_{\widetilde{\sigma}}\xrightarrow{\partial}  \bigoplus_{\substack{ \widetilde{\tau}\in \widetilde{S_{m-1}} \\ \widetilde{\tau}\subseteq U_{\widetilde{l}} }} \C \cdot  e_{\widetilde{\tau}}\rightarrow \cdots
    \end{align*}
    that are inverse to each other. 

    \smallskip
    
    It then suffices to check the compatibility with the differentials. Fix a monomial $\prod x_i^{a_i} e_{\sigma}$ and recall the definition of $d$ from Section \ref{subsec-affine-HHL}. Consider the set $\{\widetilde{\tau_1},\cdots,\widetilde{\tau_r}\}$ of all facets of $\widetilde{\sigma}$ whose image under the quotient map is a fixed facet $\tau$ of $\sigma$. Denote $\epsilon_{i}^{(j)}:=\left\lceil H_{i}(\widetilde{\sigma})  \right\rceil - \left\lceil H_{i}(\widetilde{\tau_j})  \right\rceil $ for all $i=1,\cdots,n$, the monomial $\prod x_i^{a_i} e_{\sigma}$ is mapped to $\left(\sum_{j=1}^r \prod x_i^{a_i+\epsilon_{i}^{(j)}}\right) e_{\tau}$. Now denote the image of $\prod x_i^{a_i+\epsilon_{i}^{(j)}}e_{\tau}$ by $\widehat{\tau_j}$. Then by definition we have
    \begin{align*}
        \sum_{i=1}^n (a_i +\epsilon_{i}^{(j)} )v_i^* + \sum_{i=1}^n \lceil H_i(\widehat{\tau_j}) \rceil v_i^* = \sum_{i=1}^n l_i v_i^*.
    \end{align*}
    A comparison with \eqref{eq1} yields $\left\lceil H_{i}(\widetilde{\tau_j})  \right\rceil = \left\lceil H_{i}(\widehat{\tau_j})  \right\rceil$ for all $i=1,\cdots,n$. This forces $\widetilde{\tau_j} = \widehat{\tau_j}$, which proves the compatibility.
\end{proof}

Since $U_{\widetilde{l}}$ is either empty or contractible, it suffices to determine for which $l\in M$ the region is non-empty. Let $C$ be the cone region in $M$ consisting of points whose preimages in $L^*$ lie in the dual cone $\sigma^{\vee}=\sum_{i=1}^n \R_{\geq 0}v_i^*\subseteq L^*_{\R}$. Note that $C$ is a cone only in a generalized sense since $M$ may have torsion. We will use the following form of the Farkas' Lemma from linear optimization, see e.g. \cite[Example 2.26]{BV}.

\begin{lemma}[Farkas' Lemma]\label{lem-farkas}
    Let $A\in M_{m\times n}(\R)$, $x=(x_1,\cdots,x_n)^{\top}$, $b\in\R^m$ and $Ax\leq b$ be a system of linear inequalities. Then it has a solution if and only if $b$ lies in the region
    \begin{align*}
        \Omega = \left\{Ax+s: x\in\R^n, s\in \R^m_{\geq 0}\right\}.
    \end{align*}
\end{lemma}

\begin{lemma}\label{lem-criterion-for-nonempty}
    The region $U_{\widetilde{l}}$ is non-empty if and only if $l\in C\cap M$.
\end{lemma}
\begin{proof}
    We take $A$ to be the matrix whose row vectors are $\psi(v_i)$ for $i=1,\cdots,n$, and $b$ to be $(l_1,\cdots,l_n)$. The region $\Omega$ is exactly the preimage of the cone region $C$ under the map $L^*_{\R}\rightarrow M_{\R}^{\oplus |M_{\mathrm{tor}}|}$.
\end{proof}

\smallskip

In general, let $M = M_{\mathrm{free}}\oplus M_{\mathrm{tor}}$ be a finitely generated abelian group. We say $C\subseteq M_{\R}^{\oplus|M_{\mathrm{tor}}|}$ is a cone region if for each factor $M_{\R}$, the projection of $C$ is a cone region in the usual sense. We call the associated affine scheme $\operatorname{Spec}\C[C\cap M]$ a \textit{generalized affine toric variety}. In our case all components of the cone region $C$ are identical, therefore $\operatorname{Spec}\C[C\cap M]$ is a disjoint union of isomorphic toric subvarieties. Note that the quotient map $L^*\rightarrow M$ induces a natural morphism from it to the affine space $\C^n$ whose image is not necessarily normal.

\begin{theorem}\label{thm-1}
    The only nonzero homology of the affine HHL complex \eqref{eq-affine-HHL-complex} is $H_0=\C[C\cap M]$. 
\end{theorem}
\begin{proof}
    The higher homology vanish and the 0-th homology is isomorphic to $\C[C\cap M]$ as $C$-vector spaces by Proposition \ref{prop-isomorphic-to-singular} and Lemma \ref{lem-criterion-for-nonempty}. It is straightforward to see the compatibility with module structures.
\end{proof}

\begin{remark}\label{rmk-BE-conjecture}
    It was pointed out by Michael Brown, Daniel Erman and Andrew Hanlon that the main result in this paper settles a conjecture of Brown and Erman \cite[p.3]{BE}. More precisely, our main result implies that the HHL complex, when viewed as a complex of graded modules over the homogeneous coordinate ring $S$, is a resolution of the normalization of $S/I$ where $I$ is the defining ideal of the subvariety $Y$. This shows that the resolution constructed in \cite{BE} is the minimization of the HHL complex.
\end{remark}

\begin{remark}
We observe that our proof of Theorem \ref{thm-1} does not use that our base field is $\C$. In fact, it even works over $\Z$. However, we choose to focus on the case of $\C$ for cultural reasons.
\end{remark}

\subsection{Examples}

\begin{example}[Non-normal toric subvariety $(x^3-y^2)\subseteq \mathbb{A}^2$]

    Consider $L = \Z^2$ with the standard basis, and the map $L\rightarrow\Lambda = \Z$ defined by $(1,0)\mapsto 3$ and $(0,1)\mapsto -2$. In this example the real torus is 1-dimensional and the stratification consists of 4 points and 4 segments.
    \begin{figure}[H]
        \centering
    \begin{tikzpicture}[x=4cm, y=1cm] 
        \draw (-0.1,0) -- (1.1,0);

    \draw (0,0.06)--(0,-0.06)       node[below=2pt] {$0$};
    \draw (1/3,0.06)--(1/3,-0.06)   node[below=2pt] {$\tfrac{1}{3}$};
    \draw (1/2,0.06)--(1/2,-0.06)   node[below=2pt] {$\tfrac{1}{2}$};
    \draw (2/3,0.06)--(2/3,-0.06)   node[below=2pt] {$\tfrac{2}{3}$};
    \draw (1,0.06)--(1,-0.06)       node[below=2pt] {$1$};

    \foreach \p in {0,1/3,1/2,2/3, 1}{
        \fill (\p,0) circle (2.2pt);
    }
    \end{tikzpicture}
    \hspace{1.5cm}
    \begin{tikzpicture}[scale=1.2] 
        \foreach \start in {0,-90,-180,-270}{
        \draw[midarrow] (\start:1)
        arc[start angle=\start, delta angle=-90, radius=1];
    }

    \foreach \ang in {0,90,180,270}{
        \fill (\ang:1) circle (1.5pt);
    }
    \end{tikzpicture}
    \caption{The Bondal stratification on the universal cover and the torus}
    \end{figure}
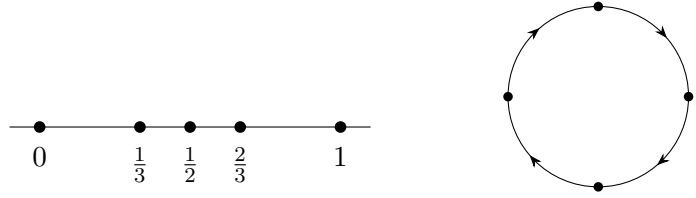
    The associated HHL complex is 
    \begin{align*}
        0\rightarrow 
        R^4\xrightarrow{\begin{pmatrix} -1 & 0 & 0 & x \\ x & -y & 0 & 0 \\ 0 & 1 & -1 & 0 \\ 0 & 0 & x & -y \end{pmatrix}} R^4 \rightarrow 0 
    \end{align*}
    where $R=\C[x,y]$ is the coordinate ring. It is straightforward to compute the homology at the last spot to be the quotient of $R^2$ by $(y,x)^{\top}$ and $(x^2,-y)^{\top}$, which is isomorphic to the normalization $\C[t]$ of the coordinate ring $\C[x,y]/(x^3-y^2)$ of the cuspidal curve.
\end{example}

\begin{example}[A disjoint union of toric subvarieties]\label{eg-torsion}
Now consider again $L = \Z^2$ with the fan $\Sigma$ given by the first orthant, but this time the map $L\rightarrow\Lambda = \Z^2$ defined by $(1,0)\mapsto (2,-1)$ and $(0,1)\mapsto (-1,2)$. The stratification of the 2-dimensional real torus has three 2-cells, six 1-cells and three 0-cells.

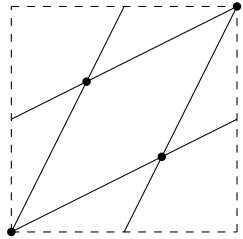
\begin{figure}[H]
    \centering
    \begin{tikzpicture}[x=3cm,y=3cm] 
        \draw[dashed] (0,0) -- (1,0);
        \draw[dashed] (0,0) -- (0,1);
        \draw[dashed] (0,1) -- (1,1);
        \draw[dashed] (1,0) -- (1,1);
        \draw (0,0) -- (0.5,1);
        \draw (0,0) -- (1,0.5);
        \draw (0,0.5) -- (1,1);
        \draw (0.5,0) -- (1,1);
        \foreach \p in {(0,0),(1,1),(1/3,2/3),(2/3,1/3)}{\fill \p circle (1.6pt);}
\end{tikzpicture}
\caption{The stratification on $[0,1]^2$}
\end{figure}

In this example $M=L^*/\Lambda^*\cong \Z/3\Z$ is torsion. The associated HHL complex is
\begin{align}\label{eq-torsion}
    0\rightarrow 
        R^3\xrightarrow{\begin{pmatrix} -x & 0 & 1 \\ y & -1 & 0 \\ 0 & 1 & -x \\ -1 & 0 & y \\ 1 & -x & 0 \\ 0 & y & -1 \end{pmatrix}} R^6 \xrightarrow{\begin{pmatrix} 0 & 1 & 1 & 0 & -y & -x \\ -y & -x & 0 & 1 & 1 & 0 \\ 1 & 0 & -y & -x & 0 & 1 \end{pmatrix}} R^3 \rightarrow 0 
\end{align}
The $0$-th homology is isomorphic to $\C[\Z/3\Z]$, which is a direct sum of three copies of $\C$.

\end{example}

\section{HHL complexes on smooth toric stacks}\label{sec-general-HHL}

First of all note that in the last section we were actually considering the nonnegative orthant fan of $L_{\R}$ (i.e., the fan consisting of all faces of $\sum_{i=1}^n \R_{\geq0} v_i$). The result can be easily generalized by considering arbitrary subfan $\Sigma$ of the nonnegative orthant fan. Now we consider a triple $(L,\Sigma,\psi: L\rightarrow \Lambda)$ where $L$ and $\psi$ are defined as before and $\Sigma$ is an arbitrary subfan of the nonnegative orthant fan. Note that it defines a toric variety $U_{\Sigma}$ which is the open subset of $\C^n$ of points whose set of zero coordinates lies in a cone of $\Sigma$.

\begin{theorem}\label{thm-2}
    Let $\psi:L\rightarrow\Lambda$ be a map of lattices with finite cokernel, and $\Sigma$ be a subfan of the first orthant fan. Then the corresponding HHL complex
    \begin{align}\label{eq-local-HHL-complex}
        0\rightarrow\bigoplus_{\substack{\sigma\in S_{k} }} \mathcal{O}_{U_{\Sigma}}\, e_{\sigma} \rightarrow \cdots \rightarrow \bigoplus_{\substack{\sigma\in S_{1}}} \mathcal{O}_{U_{\Sigma}}\, e_{\sigma} \rightarrow \bigoplus_{\substack{\sigma\in S_{0}}} \mathcal{O}_{U_{\Sigma}}\, e_{\sigma}\rightarrow0
    \end{align}
    is the resolution of the restriction of $i_*\mathcal{O}_{Y}$ to the open subset $U_{\Sigma}\subseteq\C^n$, where $Y$ is the generalized affine toric variety defined as $\operatorname{Spec}\C[C\cap M]$ and $i:Y\rightarrow\C^n$ is the natural morphism.
\end{theorem}
\begin{proof}
    The corresponding HHL complex is the localization of the affine HHL complex on the open subset $U_{\Sigma}$ described above. The result follows from the fact that taking homology commutes with localization.
\end{proof}

We are ready to generalize the result to toric stacks. We will need additional data. Consider a quadruple $(L,\Sigma,\psi: L\rightarrow \Lambda, G)$, where
\begin{itemize}
    \item $L\cong \Z^n$ is a lattice of rank $n$,
    \item $\Sigma$ is a subfan of the first orthant fan,
    \item $\psi:L\rightarrow\Lambda$ is a morphism from $L$ to another lattice $\Lambda$ of rank $k$ such that $\mathrm{coker}\,\psi$ is finite.
    \item $G$ is an algebraic group that maps to $\mathrm{ker}(L\otimes\C^*\rightarrow\Lambda\otimes\C^*)$. 
\end{itemize}
Note that the algebraic group $\mathrm{ker}(L\otimes\C^*\rightarrow\Lambda\otimes\C^*)$ is exactly $\operatorname{Spec}\C[M]$ and is a product of an algebraic torus with a finite abelian group. Any algebraic group $G$ that maps to it naturally acts on the HHL complex as well as its homology, which restricts to a $G$-equivariant structure. Denote the quotient stack $[U_{\Sigma}/G]$ by $\mathcal{U}_{\Sigma,G}$, then by the standard equivalence (see \cite{Vistoli})
\begin{align}\label{eq-equivalence}
    \mathrm{Coh}^{G}(U_{\Sigma})\xrightarrow{\sim}\mathrm{Coh}(\mathcal{U}_{\Sigma,G})
\end{align}
we obtain a complex of direct sums of line bundles on the toric stack $\mathcal{U}_{\Sigma,G}$. The line bundles in this complex are from the so-called \textit{Thomsen-Bondal collection}, see \cite[Section 2.3]{HHL}.

\smallskip

\begin{definition}
    Let $\sigma \in S$ be a stratum in the Bondal stratification of the real torus $T$. Let $\widetilde{\sigma}\in \widetilde{S}$ be an arbitrary lift of $\sigma$. We define
    \begin{align*}
        \mathcal{O}_{\mathcal{U}_{\Sigma,G}}(\sigma):=\mathcal{O}_{\mathcal{U}_{\Sigma,G}}\left( - \sum_{i=1}^n \left\lceil H_{i}(\widetilde{\sigma})  \right\rceil D_{i} \right).
    \end{align*}
    Note that the isomorphism class of the bundle does not depend on the choice of the lift $\widetilde{\sigma}$. We call the set of isomorphism classes of such line bundles  the \textit{Thomsen-Bondal collection} of $\mathcal{U}_{\Sigma,G}$.
\end{definition}

\begin{definition}
    We call the complex of line bundles on $\mathcal{U}_{\Sigma,G}$ 
    \begin{align}\label{eq-HHL-complex}
        0\rightarrow\bigoplus_{\substack{\sigma\in S_{k} }} \mathcal{O}_{\mathcal{U}_{\Sigma,G}}(\sigma) \rightarrow \cdots \rightarrow \bigoplus_{\substack{\sigma\in S_{1}}} \mathcal{O}_{\mathcal{U}_{\Sigma,G}}(\sigma) \rightarrow \bigoplus_{\substack{\sigma\in S_{0}}} \mathcal{O}_{\mathcal{U}_{\Sigma,G}}(\sigma)\rightarrow0
    \end{align}
    that corresponds to the affine HHL complex under the equivalence \eqref{eq-equivalence} the \textit{HHL complex} associated to the input data $(L,\Sigma,\Lambda,\psi: L\rightarrow \Lambda, G)$.
\end{definition}

\smallskip

\begin{theorem}\label{thm-3}
    The HHL complex \eqref{eq-HHL-complex} on $\mathcal{U}_{\Sigma,G}$ is a resolution of $i_*\mathcal{O}_{\mathcal{Y}}$, where $\mathcal{Y}$ is the quotient stack $[Y/G]$ and $i:\mathcal{Y}\rightarrow \mathcal{U }_{\Sigma,G}$ is induced by the corresponding maps between varieties.
\end{theorem}

\begin{example}
    In Example \ref{eg-torsion} the kernel of $L\otimes\C^*\rightarrow \Lambda\otimes\C^*$ is isomorphic to $\Z/3\Z$. If we take $G = 1$ then we are in the genuine variety cases described therein. If we take $G=\Z/3\Z$, then the complex \eqref{eq-torsion} descends to a complex of line bundles on the quotient stack $[\C^2/(\Z/3\Z)]$ which resolves the structure sheaf of a point.
\end{example}

\begin{remark}
The definition of toric stacks we used in this paper is close to but more general than the one introduced in \cite{BCS},
for instance it allows for non-separated smooth toric stacks. We also allow some "phantom generators", as in the work of Jiang  \cite{Jiang}, in the sense we don't have to assume that all rays of the first orthant in $L$ are cones in $\Sigma$.
\end{remark}

\smallskip

We end this paper by a comparison between our setting and the original one in \cite{HHL}. The definition of the toric stacks used in \cite{HHL} is the one introduced by Geraschenko and Satriano \cite{GS}. More precisely, the input data is $(L,\Sigma,\beta:L\rightarrow N)$ where $L$, $N$ are lattices, $\Sigma$ is a fan on $L$ and $\beta$ a lattice map with finite cokernel. The associated toric stack $\mathcal{X}_{\Sigma,\beta}$ is then defined as the quotient stack $[X_{\Sigma}/G_{\beta}]$, where $G_{\beta}$ is the kernel of $L\otimes\C^*\rightarrow N\otimes\C^*$ induced by $\beta$. They start with a closed embedding of toric stacks $\mathcal{Y}\hookrightarrow \mathcal{X}$ coming from a diagram
\begin{align*}
    \xymatrix{
        L_Y \ar@{^(->}[r] \ar[d]^{\beta_Y} & L_X \ar[d]^{\beta_X} & \\
        N_Y \ar@{^(->}[r]^{\phi} & N_X 
    }
\end{align*}
together with smooth fans $\Sigma_Y$ and $\Sigma_X$ on $L_Y$ and $L_X$ respectively. Here $\beta_Y$ is the restriction of $\beta_X$ to $L_Y$, and the fan $\Sigma_Y$ is necessarily obtained by intersecting cones of $\Sigma_X$ with the subspace $(L_Y)_{\R}$. The real torus is then defined as $\operatorname{Hom}(\mathrm{coker}\,\phi, \R/\Z)$ and the HHL complex is built from the stratification on it. Their main result states that the HHL complex on $\mathcal{X}=[X_{\Sigma}/G_{\beta_X}]$ is a resolution of the pushforward of the structure sheaf of $\mathcal{Y}=[Y_{\Sigma}/G_{\beta_Y}]$.

\smallskip

To connect the setup of \cite{HHL} to Theorem \ref{thm-3}, we define 
$L = \Z^{\Sigma(1)}$ and
\begin{align*}
    \widetilde{\Sigma} = \left\{ \mathrm{Cone}(e_{\rho}:\rho\in \sigma) : \sigma\in\Sigma \right\}.
\end{align*}
Then the toric variety $X_\Sigma$ can be obtained as the quotient of $U_{  \widetilde{\Sigma} }$ by the appropriate group,
via the homogeneous coordinate ring construction of Cox \cite{Cox}. In particular, $L$ naturally maps to $L_X$. We then take $\Lambda = \mathrm{coker}\,\phi$ (which is free by an assumption of \cite{HHL}), and define the lattice map $\psi:L\to \Lambda$ to be the composition of the maps
 $$L\rightarrow L_X\rightarrow N_X\rightarrow \mathrm{coker}\,\phi.$$
One can then view the sheaves on $X_\Sigma$ as the sheaves on $U_\Sigma$ with the equivariant structure for the group $G= \mathrm{ker}\,(L \otimes \C^* \rightarrow N_X \otimes \C^*)$, which is a subgroup of  $\mathrm{ker}(L\otimes\C^*\rightarrow\Lambda\otimes\C^*)$.
The HHL complex defined by the input data in \cite{HHL} then matches with the complex defined in this paper with input data $(L,\widetilde{\Sigma},\psi,G)$.

\end{document}